\documentclass[12pt]{amsart}

\usepackage{amsmath,amstext,amssymb,amsopn,amsthm}
\usepackage{url}
\usepackage{mathtools}
\usepackage{enumerate}
\usepackage{comment}
\usepackage{color,graphicx}
\usepackage[margin=25mm]{geometry}
\usepackage{eucal,mathrsfs,dsfont} 

\newtheorem{theorem}{Theorem}[section]

\newtheorem{lemma}[theorem]{Lemma}
\newtheorem{proposition}[theorem]{Proposition}

\theoremstyle{definition}

\newtheorem{definition}[theorem]{Definition}

\newtheorem{assumption}[theorem]{Assumption}
\newtheorem{example}[theorem]{Example}
\newtheorem{remark}[theorem]{Remark}
\numberwithin{equation}{section}

\newcommand{\eps}{\varepsilon}

\newcommand{\calL}{\mathcal{L}}

\newcommand{\calF}{\mathcal{F}}
\newcommand{\sF}{\mathcal{F}}

\newcommand{\sD}{\mathcal{D}}

\newcommand{\calO}{\mathcal{O}}
\newcommand{\sO}{\mathcal{O}}
\newcommand{\calG}{\mathcal{G}} 
\newcommand{\sG}{\mathcal{G}}
\newcommand{\calA}{\mathcal{A}}
\newcommand{\calD}{\mathcal{D}}
\newcommand{\calP}{\mathcal{P}}

\newcommand{\calS}{\mathcal{S}}
\newcommand{\sS}{\mathcal{S}}

\newcommand{\calB}{\mathcal{B}}

\newcommand{\calM}{\mathcal{M}}

\newcommand{\sE}{\mathcal{E}} 
\newcommand{\sQ}{\mathcal{Q}} 
\newcommand{\sH}{\mathcal{H}}
\newcommand{\sI}{\mathcal{I}}
\newcommand{\calK}{\mathcal{K}}

\newcommand{\I}{\mathds{1}}

\newcommand{\bE}{\operatorname{\mathbb{E}}} 
\newcommand{\Pp}{\operatorname{\mathbb{P}}} 
\newcommand{\bP}{\operatorname{\mathbb{P}}} 
\newcommand{\bD}{\operatorname{\mathbb{D}}} 

\newcommand{\R}{\mathbb{R}}
\newcommand{\bR}{\mathbb{R}}
\newcommand{\bH}{\mathbb{H}}

\newcommand{\N}{{\mathds{N}}}
\def\bN{\N}
\newcommand{\Z}{{\mathbb Z}}
\newcommand{\bZ}{{\mathbb Z}}

\newcommand{\PBM}{P_{\text{BM}}}
\newcommand{\EBM}{E_{\text{BM}}}

\newcommand{\wh}{\widehat}
\newcommand{\wt}{\widetilde}

\def\dP{{d_P}}

\def\qP{{P}}
\def\qE{{E}}
\def\pd{{\partial}}
\def\q{\quad}

\def\be{\begin{equation}}
\def\ee{\end{equation}}
\def\bes{\begin{equation*}}
\def\ees{\end{equation*}}

\def\om{{\omega}}
\def\half{\frac12}

\def\lam{\lambda}
\def\bfP{{\bf P}}

\def\sK{\calK}
\def\sA{\calA}
\def\ip#1{{\langle #1 \rangle}}

\def\sms{\smallskip}
\def\sm{\smallskip \noindent}

\def\be{\begin{equation}}
\def\ee{\end{equation}}
\def\bes{\begin{equation*}}
\def\ees{\end{equation*}}
\def\ip#1{{\langle #1 \rangle}}

\def\bfP{{\bf P}}
\def\nn{\nonumber}
\def\eqd{{\buildrel (d) \over =}}

\title[Invariance principle]{Comparison of quenched and annealed invariance 
principles for random conductance model}
\author{Martin Barlow, Krzysztof Burdzy \and \'Ad\'am Tim\'ar}

\address{Department of Mathematics, University of British Columbia, 
Vancouver, B.C., Canada V6T 1Z2}
\email{barlow@math.ubc.ca}

\address{Department of Mathematics, Box 354350, University of Washington, 
Seattle, WA 98195, USA}
\email{burdzy@math.washington.edu}

\address{
Bolyai Institute, University of Szeged,
Aradi v. tere 1, 6720 Szeged,
Hungary
}
\email{madaramit@gmail.com
}

\thanks{Research supported in part by NSF Grant DMS-1206276, by 
NSERC, Canada, and Trinity College, Cambridge,  and by
MTA R\'enyi "Lendulet" Groups and Graphs Research Group.}

\begin{document}

\begin{abstract}
We show that there exists an ergodic conductance environment such that the 
weak (annealed) invariance principle holds for the corresponding  
continuous time random walk but the quenched invariance principle does not hold.
\end{abstract}

\maketitle

\section{Introduction}\label{intro}

Let $d\geq 2$ and let $ E_d$ be the set of all non oriented edges in the 
$d$-dimensional integer lattice, that is,  $E_d = \{e = \{x,y\}: x,y \in \Z^d, |x-y|=1\}$.
Let $\{\mu_e\}_{e\in E_d}$ be a random process with non-negative values, 
defined on some probability space $(\Omega, \calF, \Pp)$. 
The process $\{\mu_e\}_{e\in E_d}$ represents random conductances.  
We write $\mu_{xy}  = \mu_{yx} = \mu_{\{x,y\}}$ and set 
$\mu_{xy}=0$ if $\{x,y\} \notin E_d$. Set
\begin{align*}
\mu_x = \sum_y \mu_{xy}, \qquad P(x,y) = \frac{\mu_{xy}}{\mu_x},
\end{align*}
with the convention that $0/0=0$ and $P(x,y)=0$ if $\{x,y\} \notin E_d$. 
For a fixed $\omega\in \Omega$, let  
$X = \{X_t, t\geq 0, \qP^x_\omega, x \in \Z^d\}$ be the 
continuous time random walk on $\Z^d$, with transition probabilities 
$P(x,y) = P_\omega(x,y)$, and exponential waiting times with mean $1/\mu_x$. 
The corresponding expectation will be denoted $\qE_\omega^x$. 
For a fixed $\omega\in \Omega$, the generator $\calL$ of $X$ is given by
\begin{align}\label{e:Ldef}
\calL f(x) = \sum_y \mu_{xy} (f(y) - f(x)).
\end{align}
In \cite{BD} this is called the {\em variable speed random walk} (VSRW)
among the conductances $\mu_e$.
This model, of a reversible (or symmetric) random walk in a random environment, is
often called the Random Conductance Model. 

We are interested in functional Central Limit Theorems (CLTs)
for the process $X$. Given any process $X$,
for $\eps>0$, set $X^{(\eps)}_t = \eps X_{t /\eps^2}$, $t\geq 0$. 
Let $\calD_T = D([0,T], \R^d)$ denote the Skorokhod space, 
and let $\calD_\infty=D([0,\infty), \R^d)$.
Write $d_S$ for the Skorokhod metric and $\calB(\calD_T)$ for the $\sigma$-field of 
Borel sets in the corresponding topology. 
Let $X$ be the canonical process on $\calD_\infty$ or $\calD_T$, $\PBM$ be Wiener 
measure on $(\calD_\infty, \calB(\calD_\infty))$  and let $\EBM$ be the 
corresponding expectation. 
We will
write $W$ for a standard Brownian motion.
It will be convenient to assume that $\{\mu_e\}_{e\in E_d}$ are 
defined on a probability space $(\Omega, \sF, \bP)$, and that
$X$ is defined on $(\Omega, \sF) \times (\calD_\infty, \calB(\calD_\infty))$ 
or  $(\Omega, \sF) \times (\calD_T, \calB(\calD_T))$. 
We also define the averaged or annealed measure $\bfP$ on 
$(\calD_\infty, \calB(\calD_\infty))$ or  $(\calD_T, \calB(\calD_T))$ by
\bes
 \bfP(G) = \bE P^0_\om(G). 
\ees

\begin{definition}\label{j1.2}
For a bounded function $F$ on $\calD_T$ and a constant matrix $\Sigma$, let 
$\Psi^F_\eps = \qE^0_\omega F(X^{(\eps)})$ and 
$\Psi^F_\Sigma = \EBM F(\Sigma W)$. 

\sm (i) We say that the {\em Quenched Functional CLT} (QFCLT) holds 
for $X$ with limit $\Sigma W$ if for every $T>0$ and 
every bounded continuous function $F$ on $\calD_T$ we 
have $\Psi^F_\eps \to \Psi^F_\Sigma$ as $\eps\to 0$, with $\Pp$-probability 1.\\
(ii) We say that the {\em Weak Functional CLT} (WFCLT) 
holds for $X$ with limit $\Sigma W$ if for every $T>0$ and every 
bounded continuous function $F$ on $\calD_T$ we have 
$\Psi^F_\eps \to \Psi^F_\Sigma$ as $\eps\to 0$, in $\Pp$-probability.\\
(iii) We say that the {\em Averaged (or Annealed) Functional CLT}
(AFCLT) holds for $X$ with limit $\Sigma W$ if for every $T>0$ and every 
bounded continuous function $F$ on $\calD_T$ we have 
$ \bE \Psi^F_\eps \to \Psi_{\Sigma}^F$.
This is the same as standard weak convergence with respect to the probability measure $\bfP$. 
\end{definition}

If we take $\Sigma$ to be non-random then since $F$ is bounded, it is
immediate that QFCLT $\Rightarrow$ WFCLT $\Rightarrow$ AFCLT.
In general for the QFCLT the matrix
$\Sigma$ might depend on the environment $\mu_\cdot(\om)$. However, if
the environment is stationary and ergodic, then $\Sigma$ is a shift invariant
function of the environment, so must be $\bP$--a.s. constant.

In \cite{DFGW} it is proved that if $\mu_e$ is a stationary ergodic 
environment with $\bE \mu_e<\infty$ then the WFCLT holds.
It is an open question as to whether the QFLCT holds under these hypotheses. 
For the QFCLT in the case of percolation see \cite{BeB, MP, SS}, and for the Random Conductance Model
with $\mu_e$ i.i.d see \cite{BP, M1, BD, ABDH}.
In the i.i.d. case the QFCLT holds (with $\sigma>0$)
for any distribution of $\mu_e$ provided $p_0=\bP(\mu_e=0) < p_c(\bZ^d)$.

\begin{definition}
We say an environment $(\mu_e)$ on $\bZ^d$ is {\em symmetric} if the law of  $(\mu_e)$ is 
invariant under symmetries of $\bZ^d$. 
\end{definition}

If $(\mu_e)$ is stationary, ergodic and symmetric, and the WFCLT holds with
limit $\Sigma W$ then the limiting covariance matrix $\Sigma^T \Sigma$ must also
be invariant under symmetries of $\bZ^d$, so must be a constant $\sigma\ge 0$
times the identity.

Our main result concerns the relation between the weak and quenched FCLT.

\begin{theorem}\label{T:main1}
Let $d=2$ and $p<1$.
There exists a symmetric stationary ergodic environment $\{\mu_e\}_{e\in E_2}$
with $\bE (\mu_e^p \vee \mu_e^{-p})<\infty$ 
and a sequence $\eps_n \to 0$ such that\\
(a)  the WFCLT holds for $X^{(\eps_n)}$ with limit $W$, \\
but \\
(b) the QFCLT does not hold for  $X^{(\eps_n)}$ with limit $ \Sigma W$ for any $\Sigma$.  
\end{theorem}

\begin{remark}
(1) Under the weaker condition that $\bE\mu_e^p<\infty$ and $\bE \mu_e^{-q}<\infty$ with 
$p<1$, $q<1/2$
we have the full WFCLT  for $X^{(\eps)}$ as $\eps\to 0$,
i.e., not just along a sequence $\eps_n$. However, the proof of this is very much harder
and longer than that of Theorem \ref{T:main1}(a). A sketch argument 
will be posted on the arxiv -- see \cite{BBT-A}. 
(Since our environment has  $\bE \mu_e = \infty$ we cannot use the results of \cite{DFGW}.)
We have chosen to use in this paper essentially the same environment as in \cite{BBT-A},
although for Theorem \ref{T:main1} a slightly simpler environment would have been sufficient. \\
(2) Biskup \cite{Bi} has proved that the QFCLT holds with $\sigma>0$ if $d=2$ and 
$(\mu_e)$ are symmetric and ergodic with $\bE( \mu_e \wedge \mu_e^{-1})<\infty$. \\
(3) See Remark \ref{R:Zd} for how our example can be adapted to $\bZ^d$ with $d\geq 3$; 
in that case we
have the same moment conditions as in Theorem \ref{T:main1}.\\
(4) A forthcoming paper by Andres, Deuschel and Slowik proves that the
QFCLT holds (in $\bZ^d$, $d\geq 2$) for stationary symmetric ergodic environments
$(\mu_e)$ under the conditions $\bE \mu_e^p <\infty$,
$\bE \mu_e^{-q}<\infty$, with $p^{-1}+q^{-1} <2/d$.
\end{remark}

Our second topic concerns the relation between the weak and averaged
FCLT. In general, of course, for a sequence of random variables $\xi_n$,
convergence of $\bE \xi_n$ does not imply convergence in probability.
However, under some hypotheses on the processes $X^{(n)}$ which
are quite natural in this context, we do find that the WFCLT and AFCLT  
are equivalent -- see Theorem \ref{j9.1}.

The remainder of the paper after Section 2 constitutes the proof of Theorem
\ref{T:main1}. The argument is split into several sections.  In the
proof, we will discuss the conditions listed in Definition \ref{j1.2}
for $T=1$ only, as it is clear that the same argument works for general $T>0$.

\sm {\bf Acknowledgment.}
We are grateful to Emmanuel Rio for very helpful advice, and Pierre Mathieu
and Jean-Dominique Deuschel for some very useful discussions.

\section{Averaged and weak invariance principles} 

As in the Introduction, let $(\Omega, \sF, \bP)$ be a probability space, fix some $T>0$ and let 
$\sD=\calD_T$ in this section (although we will also use $\calD_{2T}$).
Recall that $X$ is the coordinate/identity process on $\sD$. Let $C(\calD)$ be the family of all functions $F: \calD \to \R$ which are continuous in the Skorokhod topology. 

\begin{definition}
Probability measures $P^\om_n$ on $\sD$ {\em converge weakly in measure}
 to a probability measure $P_0$ on $\sD$ if for each bounded $F\in C(\sD)$,
\be \label{e:Pconv1}
 E^\om_n F(X) \to E_0 F(X) \hbox{ in  $\bP$ probability}.
 \ee
This definition is given in \cite{DFGW}. 
\end{definition}

Let $\delta_n \to 0$, let $\Lambda_n = \delta_n \bZ^d$, and let $\lambda_n$
be counting measure on $\Lambda_n$ normalized so that $\lam_n \to dx$
weakly, where $dx$ is Lebesgue measure on $\bR^d$.
Suppose that for each $\om$ and $n \ge 1$ we have Markov processes 
$X^{(n)}=(X_t, t\ge 0, P^x_{\om,n},  x \in \Lambda_n)$  with values in $\Lambda_n$. Write 
\bes
  P^{(\om, n)}_t f(x) = E^x_{\om,n} f( X_t)
\ees
for the semigroup of $X^{(n)}$. Since we are discussing weak convergence, it
is natural to put the index $n$ in the probability measures $P^x_{\om,n}$ rather
than the process; however we will sometimes abuse notation and refer to
$X^{(n)}$ rather than $X$ under the laws $(P^x_{\om,n})$. Recall that
$W$ denotes a standard Brownian motion.

\sms For the remainder of this section, we will suppose that the following Assumption holds. 

\begin{assumption} \label{a:xnpn}
 {\rm
(1)  For each $\om$, $P^{(\om, n)}_t$ is self adjoint on $L^2( \Lambda_n, \lam_n)$. \\
(2) The $\bP$ law of the `environment' for $X^{(n)}$ is stationary. More precisely,
for $x \in \Lambda_n$ there exist measure preserving maps  
$T_x : \Omega \to \Omega$ such that for all bounded measurable $F$ on $\sD_T$,
\begin{align} \label{e:tr1}
 E^x_{\om,n} F( X) &=  E^0_{T_x \om,n} F(  X+x) , \\
 \label{e:tr2}
  \bE E^0_{T_x \om,n} F( X) &= \bE E^0_{\om,n} F( X) .
 \end{align}
(3) The AFCLT holds, that is for all $T>0$ and bounded continuous $F$ on $\sD_T$,
\bes
\bE E^0_{\om,n} F(X) \to \EBM F(X).
\ees
} \end{assumption}

\def\omn{{(\om,n)}}

Given a function $F $ from $\sD_T$ to $\bR$ set
$$ F_x(w) = F(x+w), \q x \in \bR^d, w \in \sD_T.$$
Note that combining \eqref{e:tr1} and \eqref{e:tr2} we obtain
\bes
 \bE  E^x_{\om,n} F( X) = \bE E^0_{\om,n} F_x( X), \q x \in \Lambda_n.
\ees

\sms
Set
$$ \calP_t^n f(x) = \bE P^{\om,n}_t f(x). $$
Note that $\calP^{(n)}_t$ is not in general a semigroup. 
Write $K_t$ for the semigroup of Brownian motion on $\bR^d$.  Write also
\begin{align*}
   P^{(\om,n)} F(x) &= E^x_{\om,n} F(X), \\
  \calP^{(n)} F(x) &= \bE E^x_{\om,n} F(X), \\
 \sK F(x) &= E_{BM} F(x+W), \\
 U^\omn F(x) &= P^\omn F(x) - \sK F(x).
\end{align*}
Using this notation, the AFCLT states  that for $F \in C(\sD_T)$
\be \label{e:afclt}
 \calP^{(n)} F(0) \to \sK F(0).
 \ee

\begin{definition}
Fix $T>0$ and recall that $\sD=\sD_T$.
Write $d_U$ for the uniform norm, i.e.,
$$ d_U(w,w') = \sup_{0\le s\le T} | w(s)-w'(s)|. $$
Then $d_S(w,w') \le d_U(w,w')$, but the topologies given by the two metrics are
distinct.
\end{definition}

Let $\calM(\sD)$ be the set of measurable $F$ on $\sD$. A function
$F\in \calM(\sD)$ is uniformly continuous in the uniform norm on $\sD$ if there 
exists $\rho(\eps)$ with $\lim_{\eps \to 0} \rho(\eps) =0$ such that if $w, w' \in \sD_T$ with
$d_U(w,w')\le \eps$ then
\be \label{e:modcty}
 |F(w) -F(w') | \le \rho(\eps). 
 \ee
Write $C_U(\sD)$ for the set of  $F$ in $\calM(\sD)$ which are uniformly continuous 
in the uniform norm. Note that we  do not have $C_U(\sD) \subset C(\sD)$.

\sms 
Let $C^1_0(\bR^d)$ denote the set of continuously differentiable functions with compact support. 
Let $\sA_m$ be the set of $F$ such that 
\be \label{e:Fdef}
 F(w) = \prod_{i=1}^m f_i(w(t_i)),
\ee
where $0 \leq t_1 \leq \dots t_m \leq T$, $f_i \in  C^1_0(\bR^d)$, and let  $\sA = \bigcup_m \sA_m$. 

\begin{lemma} \label{L:uc}
Let $F \in \sA$. Then $F \in C_U(\sD)$. 
\end{lemma}

\begin{proof} Let $f \in \sA_m$. Choose $C\ge 2$ so that $||f_i||_\infty \le C$ and
$|f_i(x)-f_i(y)| \le C|x-y|$ for all $x,y ,i$. Then 
$$ |F(w) - F(w')| \le  m C^{m} d_U(w,w'). $$
\end{proof}

\begin{lemma} \label{L:st1}
For all $F \in \calM(\sD)$, 
\begin{align}\label{j27.1}
P^{(\om,n)} F(x)  \, &\eqd \, P^\omn F_x(0), \\ 
U^{(\om,n)} F(x)  \, &\eqd \, U^\omn F_x(0). \nonumber
\end{align}
\end{lemma}

\begin{proof} 
By the stationarity of the environment,
\bes  
  P^{(\om,n)} F(x) =  E^x_{\om,n }  F(X) =  E^0_{T_x \om,n }  F(X+x)
  =^{(d)}  E^0_{\om,n} F(X +x) =  P^{(\om,n)} F_x(0) .
 \ees
The result for $U^\omn$ is then immediate. 
\end{proof}

\begin{lemma} \label{L:uc2}
Let $F \in C_U(\sD_T)$. Then $ P^{(\om,n)} F_x(0)$, $ U^{(\om,n)} F_x(0)$,
and $\calP^{(n)} F(x)$ are uniformly continuous on $\Lambda_n$ for every $ n \in \bN $,
with a modulus of continuity which is independent of $n$.
\end{lemma}

\begin{proof} 
If $|x-y| \le \eps$ then $d_U(w+x,w+y)\le \eps$, so if $F\in C_U(\sD_T)$ and 
$\rho$ is such that \eqref{e:modcty} holds, then
$|F_x(w)-F_y(w)| \le \rho(\eps)$, and hence
\begin{align*}
| P^{(\om,n)}_t F_x(0) -  P^{(\om,n)}_t F_y(0)| &=
| E^0_{\om,n} F( x + X) -  E^0_{\om,n} F( y + X) | \\
&\le E^0_{\om,n} | F( x + X) -  F( y + X) | \le  \rho(\eps). 
\end{align*}
This implies the uniform continuity of 
 $ P^{(\om,n)} F_x(0)$ and $ U^{(\om,n)} F_x(0)$.
By \eqref{j27.1}, 
\begin{align*}  
   \calP ^{(n)} F (x)
= \bE P^{(\om,n)} F(x)  = \bE P^\omn F_x(0), 
\end{align*}
so the uniform continuity of $\calP^{(n)} F(x)$ follows from that of
$ P^{(\om,n)} F_x(0)$. 
\end{proof}

\begin{lemma}
Let $F \in \sA$. Then
\be \label{e:con2}
  \calP^{(n)} F(x) \to \sK F(x) \hbox{ for all } x \in \bR^d. 
\ee
\end{lemma}

\begin{proof}
The AFCLT (in \ref{a:xnpn}) implies that $\bP \cdot P^0_{\om, n}$ converge weakly to $P_{BM}$.
Hence the finite dimensional distributions of $X^{(n)}$ converge
to those of $W$, and this is equivalent to \eqref{e:con2}. \end{proof}

Let $C_b(\bR^d)$ denote the space of bounded continuous functions on $\bR^d$.

\begin{lemma} \label{L:conv1}
Let $F \in \sA$, and $h \in C_b(\bR^d)\cap L^1(\bR^d)$. Then
\be \label{e:conv3}
   \int h(x) \calP^{(n)} F(x) \lam_n(dx) \to    \int h(x) \sK F(x) dx .
\ee
\end{lemma}

\begin{proof} 
This is immediate from   \eqref{e:con2} and the uniform continuity proved in Lemma \ref{L:uc2}. 
\end{proof}

\sms
The next Lemma gives the key construction in this section: using the self-adjointness
of $P^{(\om,n)}_t$ we can  linearise expectations of products. A similar idea
 is used in \cite{ZP} in the context of transition densities.
 
Let $F \in \sA_m$ be given by \eqref{e:Fdef}. Set $s_j=t_m-t_{m-j}$, and let
$$ \wh F(w)  = \prod_{j=1}^{m-1} f_{m-j}( w_{s_j})  \prod_{j=1}^m f_j(w_{t_m +t_j}). $$
Note that $\wh F$ is defined on functions $w \in \sD_{2T}$ (not $\sD_T$).
Write $\ip{f,g}_n$ for the inner product in $L^2(\lam_n)$ and $\ip{f,g}$ for the inner product in $L^2(\bR^d)$.

\begin{lemma} \label{L:reverse}
With $F$ and $\wh F$ as above, 
\begin{align} \label{e:Fadj} 
 \int ( P^{(\om,n)} F(x) )^2 \lam_n(dx) &= \int ( P^\omn \wh F(x) ) f_m(x) \lam_n(dx), \\
  \int (\sK F(x) )^2 dx &= \int ( \sK \wh F(x) ) f_m(x) dx. \label{j26.1}
\end{align}
\end{lemma}

\begin{proof}
Using the Markov property of $X^{(n)}$
$$ P^\omn F(x) = E^x_{\om,n}   \prod_{j=1}^m f_j(w_{t_j}) 
= E^x_{\om,n}  \Big( \prod_{j=1}^{m-1} f_j(w_{t_j}) P^\omn_{t_m-t_{m-1}} f_m(X_{t_{m-1}} )\Big). $$
Hence we obtain
\begin{align*}
P^\omn F(x) = P^\omn_{t_1} \left( f_1 P^\omn_{t_2-t_1} 
    \left(  f_2 \dots P^\omn_{t_m-t_{m-1}}  f_m(x) \dots \right)\right).
\end{align*}
Using the self-adjointness of $P^\omn_t$ gives
\begin{align*}
  \ip{  P^\omn F,  P^\omn F}_n 
  &= \ip{ P^\omn_{t_1} f_1 P^\omn_{t_2-t_1} f_2 \dots P^\omn_{t_m-t_{m-1}}  f_m,
   P^\omn_{t_1} f_1 P^\omn_{t_2-t_1} f_2 \dots P^\omn_{t_m-t_{m-1}}  f_m}_n \\
 &= \ip{f_1  P^\omn_{t_1} P^\omn_{t_1} f_1 P^\omn_{t_2-t_1} f_2 \dots P^\omn_{t_m-t_{m-1}}  f_m,
   P^\omn_{t_2-t_1} f_2 \dots P^\omn_{t_m-t_{m-1}}  f_m}_n.
\end{align*}
Continuing in this way we obtain
\begin{align*}
  &\ip{  P^\omn F,  P^\omn F}_n \\
&= \ip{  P^\omn_{t_m-t_{m-1}} f_{m-1}  P^\omn_{t_{m-1}-t_{m-2}} f_{m-2} \dots 
 f_1  P^\omn_{t_1} P^\omn_{t_1} f_1 \dots  P^\omn_{t_m-t_{m-1}}  f_m, f_m }_n  \\
 &= \ip{ P^\omn \wh F, f_m}_n.
\end{align*}

The proof for $\sK$ is exactly the same.
\end{proof}

\begin{lemma}
Let  $F \in \sA$. Then
\be \label{e:Ec2}
 \bE \int ( P^{(\om,n)} F(x) - \sK F(x) )^2 \lam_n(dx) \to 0. 
\ee
\end{lemma}

\begin{proof}
We have 
\begin{align} \nn
\int ( &P^{(\om,n)}F (x)- \sK F(x) )^2  \lam_n(dx) 
 = \ip{ ( P^{(\om,n)}F - \sK F), ( P^{(\om,n)}F - \sK F) }_n \\
 \nn
 &= \ip{  P^{(\om,n)}F,  P^{(\om,n)}F}_n - 
 2 \ip{  P^{(\om,n)}F,  \sK F}_n + \ip{ \sK F ,  \sK F}_n.
\end{align}
Thus
\begin{align} \nn
 \bE \int &( P^{(\om,n)}F(x) - \sK F(x) )^2 \lam_n(dx) \\
\label{e:Pdec}
  &= \bE \ip{ P^{(\om,n)}F,  P^{(\om,n)}F}_n - 
 2 \ip{  \calP^{(n)}F, \sK F}_n + \ip{\sK F,  \sK F }_n. 
  \end{align}
Since $\sK F$ is continuous  we have 
$$   \ip{ \sK F, \sK F  }_n \to  \ip{ \sK F ,  \sK F}. $$
Taking $h= \sK F $ Lemma \ref{L:conv1} gives that
$$  \ip{  \calP^{(n)} F, \sK F}_n \to  \ip{ \sK F , \sK F}. $$
Let $f_m$ and $\wh F$ be as in the the previous lemma.  Then
\bes
 \bE \ip{ P^{(\om,n)}F,  P^{(\om,n)}F}_n =\bE \ip{ P^\omn \wh F, f_m}_n  
  =  \ip{ \calP^{(n)}  \wh F, f_m}_n.
\ees
Again by Lemma \ref{L:conv1} and \eqref{j26.1},
$$    \ip{ \calP^{(n)}  \wh F, f_m}_n  \to  \ip{ \sK \wh F, f_m} = \ip{\sK F, \sK F} . $$
Adding the limits of the three terms in \eqref{e:Pdec}, we obtain \eqref{e:Ec2}. 
\end{proof}

\begin{lemma} \label{L:n1}
Let $F \in \sA$. Then
\be \label{e:cp0}
  P^{(\om,n)} F(0) \to \sK F(0) \hbox{ in $\bP$-probability. }
\ee
\end{lemma}

\begin{proof}
The previous lemma gives
$$ \bE \int ( U^\omn F(x))^2  \lam_n(dx) \to 0. $$
Using Lemma \ref{L:st1} we have
\be \label{e:ufx}
 \bE \int ( U^\omn F_x(0) )^2  \lam_n(dx) \to 0,
 \ee
 and using the uniform continuity of $U^\omn F_x(0)$ gives \eqref{e:cp0}. 
\end{proof}

\sms Write $\bD$ for the set of dyadic rationals.

\begin{proposition} \label{P:fddsub}
Given any subsequence $(n_k)$ 
there exists a subsequence $(n'_k)$ of $(n_k)$  and a set $\Omega_0$
with $\bP(\Omega_0)=1$,  such that for any $\om \in \Omega_0$ and 
$q_1 \le q_2 \le \dots \le q_m$ with $q_i \in \bD$, the r.v.
$(X_{q_i}, i=1, \dots ,m)$ under $P^0_{\om,n'_k}$ converge in
distribution to $(W_{q_i}, i=1, \dots, m)$.
\end{proposition}

\begin{proof} Let  $\bD_T=[0,T] \cap \bD$.
Fix a finite set $q_1 \le  \dots \le q_m$ with $q_i \in \bD_T$. 
Then convergence of $(X_{q_i}, i=1, \dots ,m, P^0_{\om,n} )$ is determined
by a countable set of functions $F_i \in \sA_m$. So by  Lemma \ref{L:n1} we can find nested 
subsequences $(n^{(i)}_k)$ of  $(n_k)$ such that for each $i$
$$  \lim_{k \to \infty} P^0_{(\om,n^{(i)}_k)} F_j(0) = \sK F_j(0) \qquad
\hbox{ $\bP$-a.s., for $1\le j \le i$}. $$
A diagonalization argument then implies that there exists a subsequence $n''_k$ such that 
$(X_{q_i}, i=1, \dots, m, P^0_{\om,n''_k} )$ converge in distribution to $(W_{q_i}, i=1, \dots, m)$.
Since the set of the finite sets $\{q_1, \dots, q_m\}$ is countable, 
an additional diagonalization argument
then implies that there exists a subsequence $(n'_k)$ such that this
convergence holds for all such finite sets. 
\end{proof}

\begin{theorem}\label{j9.1}
Suppose that Assumption \ref{a:xnpn} holds, and that in addition $\bP$-a.s., 
$$ \{ X, P^0_{\om,n}, n \ge 1\} \hbox{ is relatively compact}. $$
Then $(X, P^0_{\om,n})$ converge weakly in measure to Brownian motion.
\end{theorem}

\begin{proof}
Let $F \in C(\sD_T)$. If \eqref{e:Pconv1} fails, then there exists $\eps>0$ and a
subsequence $(n_k)$ such that
\be \label{e:noconv}
 \bP( | E^0_{\om, n_k} F(X) - E_{BM} F(W) | > \eps ) > \eps 
\hbox{ for all } k \ge 1. 
\ee
If $(n'_k)$ is the subsequence given by Proposition \ref{P:fddsub} then by
\cite[Thm III.7.8]{EK} we have $X^{(n'_k)} \Rightarrow W$, $\bP$-a.s.,
which contradicts \eqref{e:noconv}. 
\end{proof}

\sms We conclude the section with an example which shows the difficulties
involved in proving tightness for the laws $P^0_{\om,n}$. 

\begin{example} 
{\rm 
Let $T=1$, and let $\delta_n \downarrow 0$ be strictly decreasing. For $x =x(\cdot )\in \sD_1$
recall the definition of the oscillation function $w'(x,\delta)= w'(x,\delta,1)$ from
\cite[Chapter III]{EK}. Let  
\begin{align*} 
G_1 &= \{x \in \sD_1: w'(x, \delta_1) \le 1\}, \\
G_n &= \{x \in \sD_1: w'(x, \delta_{n-1}) >1, w'(x, \delta_{n}) \le 1 \}, \qquad n\geq 2.
\end{align*}
So $(G_n)_{n\geq 1}$ form a partition of $\sD_1$, and $p_n := P_{BM}(G_n)>0$ for each $n\geq 1$. 
Define probability measures on $\sD_1$ by
$$ Q_n( H) = \PBM (H \mid G_n). $$
Now let $(\Omega, \sF, \bP)$ be a probability space carrying i.i.d.r.v.
$\xi_j$ with $\bP(\xi_j=n) =p_n$ for all $n\ge 1$, $j\ge 1$.
Set $F_{j,n} = \{ \xi_j =n \}$, and define 
\bes
 P_{\om,n} = Q_{\xi_n(\om)}.
\ees
It is easy to verify that the averaged or annealed laws $P_n = \bP \cdot P_{\om,n}$ 
all equal $P_{BM}$, so that the AFCLT holds. However, with 
$\bP$-probability one, the laws $P_{\om,n}$ are not tight.
}
\end{example}

\section{Construction of the environment}\label{const}  

The remainder of this paper is concerned with the proof of Theorem \ref{T:main1}. 
Let $\Omega = (0,\infty)^{E_2}$, and $\calF$ be the Borel $\sigma$-algebra defined 
using the usual product topology. Then every $t\in\Z^2$ defines a transformation 
$T_t (\omega)=\omega +t$ of $\Omega$. Stationarity and ergodicity of the measures 
defined below will be understood with respect to these transformations. 

All constants (often denoted $c_1, c_2$, etc.) are assumed to be strictly positive and finite.
For a set $A \subset \bZ^2$ let $E(A)$ be the set of edges in
$A$ regarded as a subgraph of $\bZ^2$. Let $E_h(A)$ and $E_v(A)$ respectively
be the set of horizontal and vertical edges in $E(A)$.
Write $x \sim y$ if $\{x,y\}$ is an edge in $\bZ^2$. Define the exterior boundary of $A$ by
$$ \pd A =\{ y \in \bZ^2 -A: y \sim x \text{ for some } x \in A \}. $$
Let also
$$ \pd_i A = \pd(\bZ^2 -A). $$ 
Finally define balls in the $\ell^\infty$ norm by $B_\infty(x,r)= \{y: ||x-y||_\infty \le r\}$; of 
course this is just the square with center $x$ and side $2r$.

Let $\{a_n\}_{n\geq 0}$,  $\{ \beta_n\}_{n \ge 1}$ and $\{b_n\}_{n\geq 1}$ be  
strictly increasing sequences of positive integers growing to infinity with $n$,
with 
$$ 1=a_0 < b_1 < \beta_1 < a_1 \ll b_2 <  \beta_2<   a_2 \ll b_3 \dots $$
We will impose a number of conditions on these sequences in the course
of the paper. We collect these conditions here so that the reader 
can check that all conditions can be satisfied simultaneously.
There is some redundancy in the conditions, for easy reference.
(Some additional conditions on $b_n/a_{n-1}$ are needed for the proof in
\cite{BBT-A} of the full WFCLT for $(X^{(\eps)})$.)
\begin{enumerate}[(i)]
\item  $a_n$ is even  for all $n$. 
\item For each $n \ge 1$, $a_{n-1}$ divides $b_n$, 
and $b_n$ divides $\beta_n$ and $a_n$. 
\item  $b_1 \geq 10^{10}$.
\item  $a_n/\sqrt{2n} \le     b_n \le a_n /  \sqrt{n} $ for all $n$, and
$b_n \sim a_n/\sqrt{n}$.
\item $b_{n+1} \ge 2^n b_n$ for all $n$.
\item $b_n > 40 a_{n-1}$ for all $n$.
\item $b_n$ is large enough so that \eqref{e:an-choose1}
and \eqref{e:bncond1} hold.
\item $100b_n < \beta_n \le  b_n n^{1/4}  < 3 \beta_n <   a_n/10$ for $n$ large enough.
\end{enumerate}
These conditions do not define $a_n$'s and $b_n$'s  uniquely.
It is easy to check that there exist constants that 
satisfy all the conditions: if $a_i,b_i,\beta_i$ have been chosen for all $i\in\{1,\ldots, n-1\}$, then 
if $b_n$ is chosen large enough (with care on respecting the divisibility condition in (ii)), it will 
satisfy all the conditions imposed on it with respect to constants of smaller indices. Then 
one can choose $a_n$ and $\beta_n$ so that the remaining conditions are satisfied.

We set
\be \label{e:mndef}
 m_n = \frac{a_n}{a_{n-1}}, \qquad 
\ell_n = \frac{a_n}{b_n}.
\ee

\smallskip\noindent
We begin our construction by defining a collection of squares in $\bZ^2$. Let
\begin{align*} 
B_n &= [0, a_n]^2,  \\
B_n' &= [0, a_n-1]^2 \cap \bZ^2 ,\\
\calS_n(x) &= \{ x + a_n y + B_n': \, y \in \bZ^2 \}.
\end{align*} 
Thus $\sS_n(x)$ gives a tiling of $\bZ^2$ by disjoint squares of side $a_n-1$
and period $a_n$.
We say that the tiling $\calS_{n-1}(x_{n-1})$ is a refinement
of $\calS_n(x_n)$ if every square $Q \in \calS_n(x_n)$ is a finite
union of squares in $\calS_{n-1}(x_{n-1})$. It is clear that 
$\calS_{n-1}(x_{n-1})$ is a refinement of $\calS_n(x_n)$ if
and only if $x_n = x_{n-1}+ a_{n-1}y$ for some $y \in \bZ^2$.

Take $\sO_1$ uniform in $B'_1$, and for $n\geq 2$
take $\sO_n$, conditional on $(\sO_1, \dots, \sO_{n-1})$, 
to be uniform in $B'_n \cap ( \sO_{n-1} + a_{n-1}\bZ^2)$. We now define random tilings by letting
\bes
 \sS_n = \sS_n( \sO_n), \, n \ge 1.  
\ees

Let $\eta_n$, $K_n$ be positive constants; we will have $\eta_n \ll 1 \ll K_n$.
We define conductances  on $E_2$ as follows. 
Recall that $a_n$ is even, and let $a_n' = \frac12 a_n$. Let
$$ C_n = \{ (x,y) \in B_n \cap \bZ^2: y \ge x, x+y \le a_n \}. $$
We first define conductances $\nu^{0,n}_e$ for $e \in E(C_n)$. Let
\begin{align*}
D_n^{00} &= \big\{ (a'_n - \beta_n,y), a'_n - 10 b_n \le y \le a'_n + 10 b_n \big\},  \\
D_n^{01} &= \big\{ (x, a'_n + 10 b_n),  (x, a'_n + 10 b_n + 1), (x, a'_n - 10 b_n),  (x, a'_n - 10 b_n -1),   \\
\nonumber   
   & \q \q \q a'_n -\beta_n -b_n \le x \le a'_n -\beta_n + b_n \big\}.
\end{align*}
Thus the set $D^{00}_n \cup D_n^{01}$ resembles the letter I (see Fig.~\ref{fig1}).

For an edge $e \in E(C_n)$ we set 
\begin{align*} 
 \nu^{n,0}_{e}  &= \eta_n \q \text {if } e \in E_v(D^{01}_n), \\
 \nu^{n,0}_{e}  &= K_n \q \text {if } e \in E(D^{00}_n), \\
  \nu^{n,0}_{e}  &= 1 \q \text {otherwise.} 
\end{align*} 
\begin{figure} \includegraphics[width=4cm]{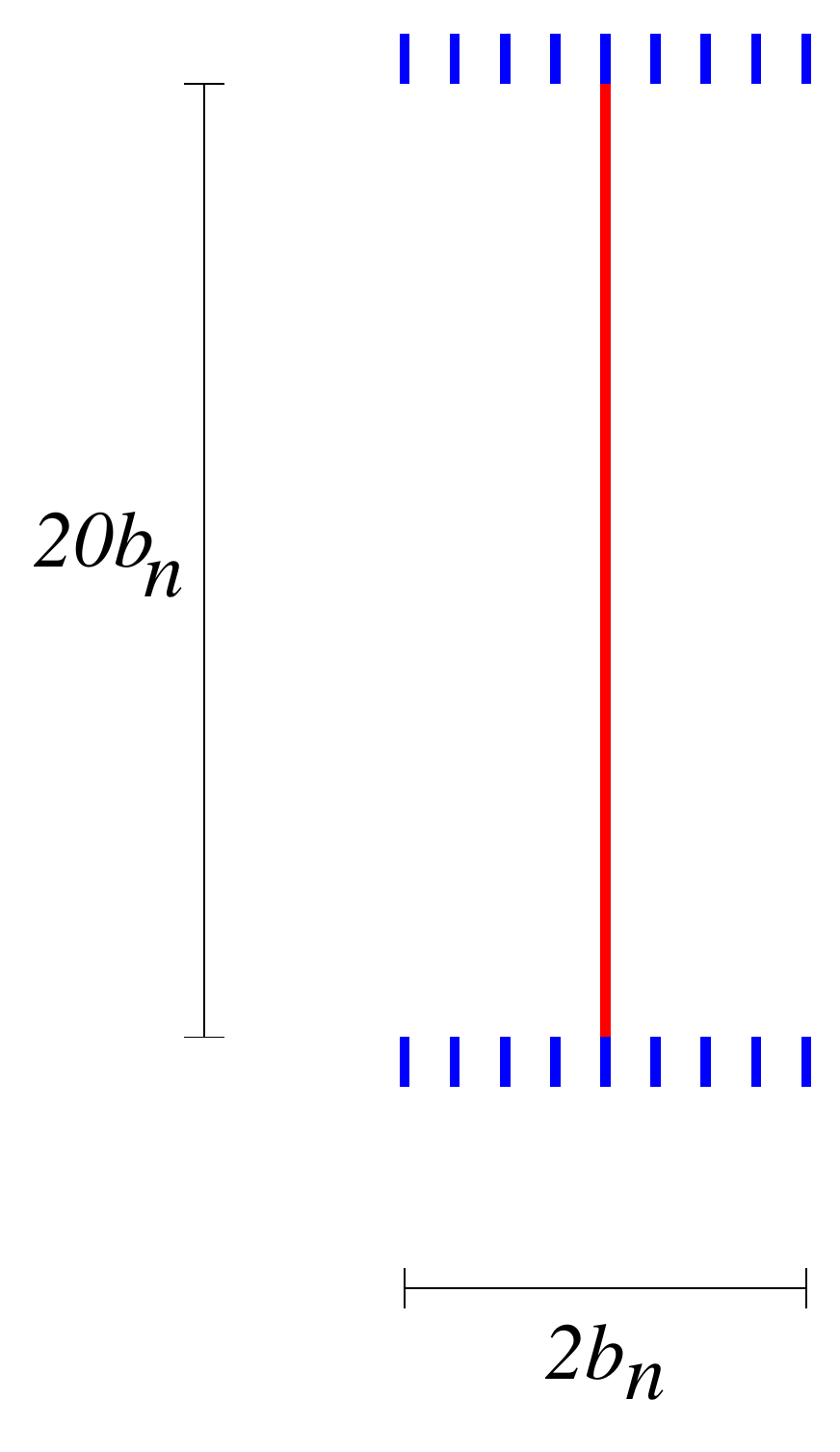}
\caption{The set $D^{00}_n \cup D_n^{01}$ resembles the letter I.
Blue edges have very low conductance. The red line represents edges with very 
high conductance. Drawing not to scale. 
}
\label{fig1}
\end{figure}

We then extend $\nu^{n,0}$ by symmetry to $E(B_n)$. More precisely,
for $z =(x,y) \in B_n$, let $R_1 z=( y,x)$ and  $R_2z = (a_n-y,a_n-x)$, so that
$R_1$ and $R_2$ are reflections in the lines $y=x$ and $x+y=a_n$.
We define $R_i$ on edges by $R_i (\{x,y\}) = \{R_i x, R_i y \}$ for $x,y \in B_n$. 
We then extend $\nu^{0,n}$ to $E( B_n)$ so that
$\nu^{0,n}_e =  \nu^{0,n}_{R_1 e }=\nu^{0,n}_{R_2 e }$ for  $e \in E(B_n)$.
We define the {\em obstacle} set $D_n^0$ by setting
(see Fig.~\ref{fig2}),
\begin{figure} \includegraphics[width=6cm]{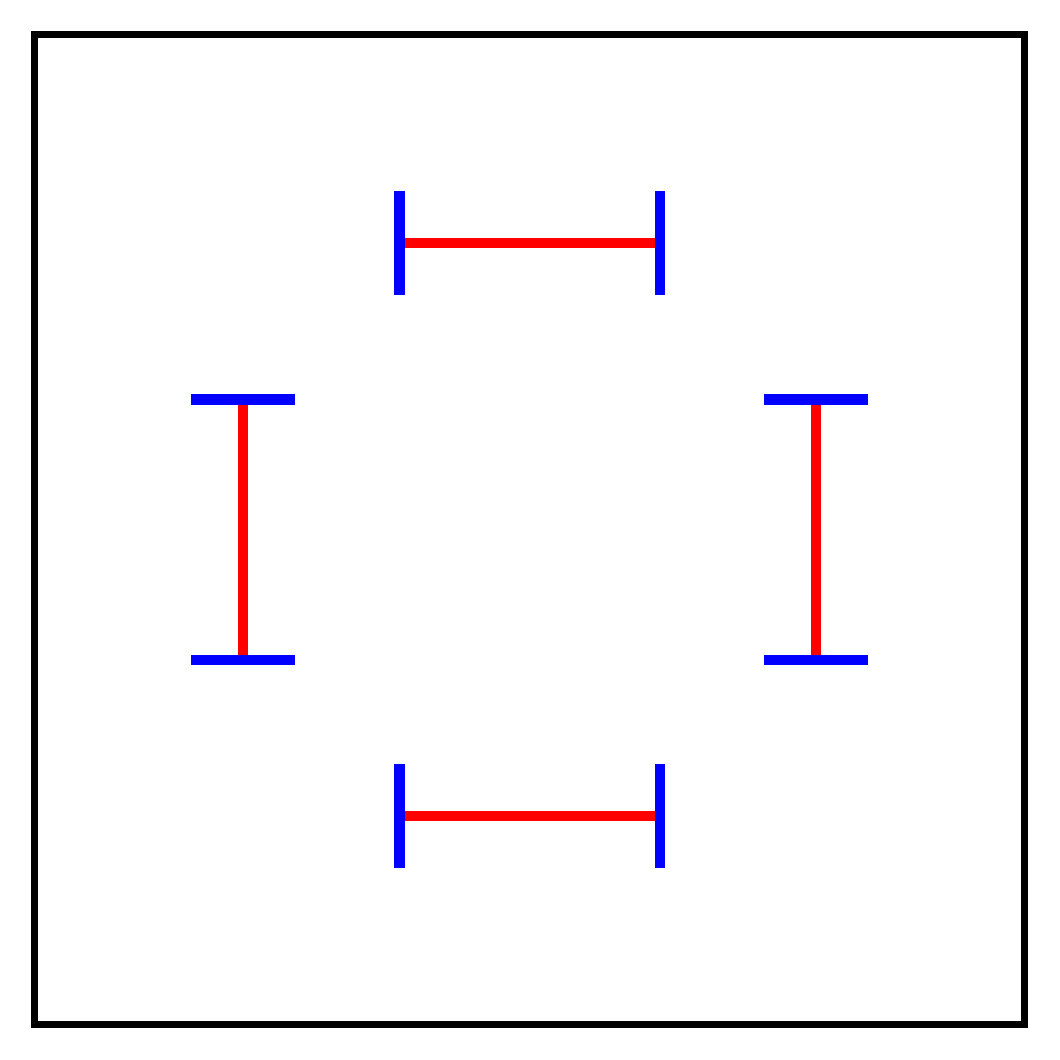}
\caption{The obstacle set $D_n^0$.
Blue lines represent ``ladders'' consisting of parallel edges with very low conductance. 
Each red line represents a sequence of adjacent edges with very high conductance.
Drawing not to scale.
}
\label{fig2}
\end{figure}
$$ D_n^0 = \bigcup_{i=0}^1 \big( D_n^{0,i} \cup R_1(D_n^{0,i})  \cup  R_2(D_n^{0,i})
 \cup R_1R_2 (D_n^{0,i} ) \big). $$
Note that $\nu^{n,0}_e=1$ for every edge adjacent to the boundary of $B_n$,
or indeed within a distance $ a_n/4$ of this boundary.
If $e=(x,y)$, we will write $e-z = (x-z,y-z)$. 
Next we extend $\nu^{n,0}$ to $E_2$ by periodicity, i.e.,
$\nu^{n,0}_e = \nu^{n,0}_{e+ a_n x}$ for all $x\in \Z^2$.
Finally, we define the conductances $\nu^n$ by translation by $\sO_n$, so that
\bes
 \nu^n_e =\nu^{n,0}_{e-\sO_n}, \, e \in E_2.
\ees
We also define the obstacle set at scale $n$ by
\bes
 D_n = \bigcup_{ x \in \bZ^2} (a_n x + \sO_n + D^0_n ).
\ees

We define the environment $\mu^n_e$ inductively by
\begin{align*}
 \mu^n_e &= \nu^{n}_e \q \text{ if } \nu^n_e \neq 1, \\
 \mu^n_e &= \mu^{n-1}_e \q \text{ if } \nu^n_e=1.
\end{align*}
Once we have proved the limit exists, we will set
\be \label{e:mudef}
 \mu_e = \lim_n \mu^n_e.
\ee

\begin{theorem} \label{T:erg}
(a) The environments $(\nu^n_e, e\in E_2)$, $(\mu^n_e, e\in E_2)$
are stationary, symmetric and ergodic.\\
(b) The limit \eqref{e:mudef} exists $\bP$--a.s. \\
(c) The environment $(\mu_e, e \in E_2)$ is  stationary, symmetric and ergodic.
\end{theorem}

\begin{proof} 
(a) The random environments $(\nu^n_e, e\in E_2)$ and  $(\mu^n_e, e\in E_2)$ are 
equivariant functions of  $(S_1(\calO_1),\ldots , S_n(\calO_n))$ (where equivariance 
of a function means that it commutes with any isometry of $\bZ^2$).
Hence, to prove the theorem for $(\nu^n_e, e\in E_2)$ and  $(\mu^n_e, e\in E_2)$ it is 
enough to show that the the family $(S_1(\calO_1),\ldots , S_n(\calO_n))$ of 
random tilings is stationary, symmetric and ergodic.
Similarly, it is enough to show the claim for the family of random subsets 
$(\calO_1+a_1\bZ^2,\ldots, \calO_n+a_n\bZ^2)$, because 
$(S_1(\calO_1),\ldots , S_n(\calO_n))$ is an equivariant function of it. 

For $x=(x_1,x_2) \in \bZ^2$ define the modulo $a$ value of $x$ as the unique
$(y_1,y_2)\in [0,a-1]^2$ such that $x_1\equiv y_1$ (mod $a$) and $x_2\equiv y_2$ (mod $a$). 
We say that $x,y\in\bZ^2$ are equivalent modulo $a$ if their modulo $a$ values are the same, 
and denote it by $x\equiv y$ mod $a$. 

Let $\calK _n$ be the set of  $n$-tuples $(x_1,\ldots , x_n)$ with
$x_i\in (x_{i-1}+a_{i-1}\bZ^2)\cap [0,a_i-1]^2$ (with the convention $a_0=1, x_0=0$). 
Denote the uniform measure on $\calK_n$ by $\Pp _n$.
Note that $(\calO_1,\ldots,\calO_n)$ is distributed according to $\Pp_n$. 

Let $U_n$ be a uniformly chosen element of $[0,a_n-1]^2\cap\bZ^2$. Then 
since each $a_{i-1}$ divides $a_i$, the distribution 
of $(U_n+a_1 \bZ^2,\ldots, U_n +a_n \bZ^2)$ is stationary, symmetric and ergodic with 
respect to the isometries $(\hat T_t, t \in \bZ^2)$ defined by
$$ \hat T_t: (U_n+a_1 \bZ^2,\ldots, U_n +a_n \bZ^2) 
\to (t+U_n+a_1 \bZ^2,\ldots, t+ U_n +a_n \bZ^2). $$
Let $\beta$ be the bijection between the set 
$\{(t+a_1\bZ^2,\ldots, t+a_n \bZ^2)\, , \, t\in [0,a_n-1]^2\cap\bZ^2\}$ and the set 
$\{(x_1+a_1\bZ^2,\ldots, x_n+a_n \bZ^2)\, , \, (x_1,\ldots, x_n)\in \calK_n\}$ given
by $\beta (t)=(x_1,\ldots , x_n)$ where $x_i$ is the mod $a_i$ value of $t$. 
The push-forward of the uniform measure for $U_n$ is then the uniform measure on 
$\calK_n$. Furthermore, $\beta$ commutes with translations. That is, if 
$\beta(t)=(x_1,\ldots , x_n)$ and $\tau\in\bZ$, then 
$\beta (t+\tau) =(x_1+\tau,\ldots , x_n+\tau)$, where addition in the $i$'th coordinate 
is understood modulo $a_i$. Similarly, $\beta$ commutes with rotations and reflections.
Hence symmetry, stationarity and ergodicity of $(O_1+a_1\bZ^2, \ldots, O_n +a_n\bZ^2)$ 
follows from that of $(U_n+a_1\bZ^2, \ldots, U_n+a_n\bZ^2)$.

\sm (b) $B_n$ contains more than $2a_n^2$ edges, of which less than $100 b_n$ are
such that $\nu^{n,0}_e\neq 1$. So by the  stationarity of $\nu^n$,
$$ \bP( \nu^n_e \neq 1) \le \frac{50 b_n}{a_n^2} \leq \frac c {2^n}. $$
The convergence in \eqref{e:mudef} then follows by the Borel-Cantelli lemma.

\sm (c) The definition \eqref{e:mudef} shows that $(\mu_e, e\in E_2)$ 
is stationary and symmetric, so all that remains to be proved is ergodicity.
Since $(\mu_e, e \in E_2)$ is an equivariant function of 
$(\calO_1+a_1\bZ^2,\calO_2+a_2\bZ^2,\ldots)$, it is enough to prove
ergodicity of the latter.

Denote by $\calK_\infty$ the family of sequences $(x_1,x_2,\ldots)$, satisfying 
$x_i\in (x_{i-1}+a_{i-1}\bZ^2)\cap [0,a_i-1]^2$ for every $i$. 
Let $\calG_\infty$ be the $\sigma$-field generated by $(\calO_1,\calO_2,\ldots)$, 
and (by a slight abuse of notation) 
for the rest of this proof let $\Pp$ be the law of $(\calO_1,\calO_2,\ldots)$.
Let $\calG_n$ be the sub-$\sigma$-field of $\calG_\infty$
generated by $(\calO_1,\ldots,\calO_n)$.

If $(x_1,x_2,\ldots)\in \calK_\infty$, $t\in\bZ^2$, define the $\Pp$-preserving transformation 
$t+(x_1,x_2,\ldots)$ as $(t+x_1,t+x_2,\ldots)$, where in the $i$'th coordinate is modulo 
$a_i$. Using the notation 
$(x_1,x_2,\ldots)+(a_1\bZ^2,a_2\bZ^2,\ldots)=(x_1+a_1\bZ^2,x_2+a_2\bZ^2,\ldots)$, 
we have 
$(t+(x_1,x_2,\ldots))+(a_1\bZ^2,a_2\bZ^2,\ldots)=t+((x_1,x_2,\ldots)+(a_1\bZ^2,a_2\bZ^2,\ldots))$. 
That is, $(\calO_1+a_1\bZ^2,\calO_2+a_2\bZ^2,\ldots)$ is an equivariant function of 
$(\calO_1,\calO_2,\ldots)$. So it is enough to prove ergodicity for $(\calO_1,\calO_2,\ldots)$.

Now let $A\in \calG_\infty$ be invariant, and suppose by contradiction 
that there is some $\eps>0$ such that $\eps < \bP (A)< 1-\eps$. There exists 
some $n$ and $B\in \calG_n$ with the property that $\bP (A\triangle B)<\eps /4$ (where $\triangle$ is the symmetric difference operator). 
This also implies that $3\eps /4<\bP(B)<1- 3\eps /4$. We have for $t \in \bZ^2$
\begin{align*}
\bP (B\triangle (B+t))&\le \bP (A\triangle B) + \bP (A\triangle (B+t)) 
= \bP (A\triangle B)  + \bP ((A+t)\triangle (B+t)) \\
&= \bP (A\triangle B) + \bP ((A\triangle B)+t) =2\bP (A\triangle B)< \eps/2.
\end{align*}
We now show that we can choose $t$ so that 
$\bP (B\triangle (B+t))  \ge 2\bP(B)\bP(\calK _\infty\setminus B)\geq \eps/2$, 
giving a contradiction. 

For an $E\in\calG_n$ denote by 
$E_n$ the subset of $\calK_n$ such that $(\calO_1,\calO_2,\ldots)\in E$ if and only if 
$(\calO_1,\ldots, \calO_n)\in E_n$. Note that $\Pp (E)=\Pp_n (E_n)$. So we want to 
show that for any $B\in\calG_n$ there exists a $t$ such that 
$\bP_n (B_n\triangle (B_n+t)) \ge 2\bP_n(B_n)\bP_n(\calK _n\setminus B_n)$.

Consider the following average:
\begin{align}\label{j27.2}
\frac{1}{a_n^2} \sum_{t\in[0,a_n-1]^2} \bP_n (B_n\triangle (B_n+t)) &=
\frac{2}{a_n^2} \sum_{t\in[0,a_n-1]^2} \bP_n (B_n\setminus (B_n+t)) \\
&=\frac{2}{a_n^4} \sum_{t\in[0,a_n-1]^2}\sum_{x\in \calK _n} \I (x\in B_n\setminus (B_n+t)). \nonumber
\end{align}
Use 
\begin{align*}
\sum_{x\in \calK _n} \I (x\in B_n\setminus (B_n+t))= \sum_{x\in B_n} \I (x\in B_n\setminus (B_n+t))=
\sum_{x\in B_n} \I (x-t\not\in  B_n)
\end{align*}
and change the order of summation 
to obtain 
\begin{align}\label{j27.3}
&\frac{2}{a_n^4}\sum_{t\in[0,a_n-1]^2}\sum_{x\in \calK _n} \I (x\in B_n\setminus (B_n+t))=
\frac{2}{a_n^4}\sum_{x\in B_n}\sum_{t\in[0,a_n-1]^2} \I (x-t\not\in  B_n)\\
&\qquad= \frac{2}{a_n^4}\sum_{x\in B_n} (a_n^2-|B_n|)
=\frac{2}{a_n^4}|B_n|(a_n^2-|B_n|)=2\bP_n(B_n)\bP_n(\calK _n\setminus B_n).\nonumber
\end{align}
It follows from \eqref{j27.2}--\eqref{j27.3} that there exists a $t\in [0,a_n-1]^2$ such that  
$\bP_n (B_n\triangle (B_n+t)) \geq 2\bP_n(B_n)\bP_n(\calK _n\setminus B_n)$.

\end{proof}

\section{Choice of $K_n$ and $\eta_n$ }\label{s:choice}  

Let 
\begin{align}\label{j27.4}
\calL_n f(x) = \sum_{y} \mu^n_{xy} (f(y)-f(x)), 
\end{align}
and $X^n$ be the associated Markov process. 

\begin{proposition} \label{P:qnclt}
For each $n \ge 1$ there exists a constant $\sigma_n$, depending only
on $\eta_i$, $K_i$, $1\le i \le n$, such that the QFCLT holds for $X^n$ with limit $\sigma_n W$.
\end{proposition}

\begin{proof} Since $\mu_e^n$ is stationary, symmetric and ergodic, and
$\mu^n_e$ is uniformly bounded and bounded away from 0, the result follows 
from \cite[Theorem 6.1]{BD}  (see also Remarks 6.2 and 6.5 in that paper). 
\end{proof}

\sms
We now set 
\be \label{e:etadef}
  \eta_n =   b_n^{-(1+1/n)}, \, n \ge 1.
\ee

\begin{remark}
For the full WFCLT proved in \cite{BBT-A} we take $\eta_n = O(a_n^2)$.
\end{remark}

\begin{theorem}\label{T:eK}
There exist constants $K_n$ such that $\sigma_n=1$ for all $n$.
\end{theorem}

\begin{proof} Let $n \ge 1$; we can assume that $K_i, 1\le i \le n-1$
have been chosen so that $\sigma_i=1$ for $i \le n-1$.
The environment $\mu^n$ is periodic, so we can use the theory of
homogenization in periodic environments (see \cite{BLP}) to calculate $\sigma_n$. 

Since $\sigma_n$ is non-random, we can simplify our notation
and avoid the need for translations by assuming that $\sO_k=0$ for $k=1, \dots n$;
note that this event has strictly positive probability.

Let $k \in \{a_{n-1}, b_n, a_n\}$, and let
$$ \sQ_k = \{ [0,k]^2 + z, z \in k\bZ^2 \}. $$
Thus $\sQ_k$ gives a tiling of $\bZ^2$ by squares of side $k$ which are disjoint
except for their boundaries. To avoid double counting of the borders, given
$Q \in \sQ_k$ and $m \in \{n-1,n\}$ set
\begin{align*}
 {\wt \mu}^{Q,m}_{xy} = 
\begin{cases}
 \half \mu^m_{xy} & \hbox{ if } x,y \in \pd_i(Q), \\
  \mu^m_{xy} & \hbox{ otherwise}.
\end{cases}
\end{align*} 
For $f: Q \to \bR$ set
$$ \sE^m_Q (f,f) = \half \sum_{x,y \in Q} \wt \mu^{Q,m}_{xy} (f(y)-f(x))^2. $$
Let $\sH_n =\{ f:B_n \to \bR \text{ s.t. } f(x,0)=0, f(x,a_n)=1, 0\le x\le a_n\}$.  Then
\be \label{e:vp}
 \sigma_n^2 = \inf\{ \sE^n_{B_n}(f,f): f \in \sH_n \}. 
\ee
Thus $\sigma_n^{-2}$ is just the effective resistance across the square $B_n$.
(Note that this would be 1 if one had $\mu^n_e \equiv 1$).
For $K \in [0,\infty)$ let $\sigma^{2}_n(K)$ be the effective conductance
across $B_n$ if we take $K_n=K$.
Since $B_n$ is finite, $\sigma_n^2(K)$ is a continuous non-decreasing function
of $K$. We will show that $\sigma_n^2(0)<1$ and $\sigma^2_n(K)>1$ for
sufficiently large $K$; by continuity it follows that there exists 
a $K_n$ such that $\sigma^2_n(K_n)=1$. 

Let $h_{n-1}$ be the function which attains the minimum in \eqref{e:vp}
for $n-1$. Note that $h_{n-1}$ is harmonic in the interior of $B_{n-1}$. 
By the inductive hypothesis we have $\sE^{n-1}_{B_{n-1}}(h_{n-1},h_{n-1})=1$. 
Further, since $\mu_e^{n-1}$ is symmetric
with respect to reflection in the axis $x_1 =a'_{n-1}$, we have
$h_{n-1}(0,x_2)=h_{n-1}(a_{n-1},x_2)$ for $0\le x_2 \le a_{n-1}$.
Let $f: B_n \to [0,1]$ be the function obtained by pasting together
shifted copies of $h_{n-1}$ in each of the squares in $\sS_{n-1}$ 
contained in $B_n$. More precisely, extend $h_{n-1}$ by periodicity
to $\bZ \times \{0, \dots a_{n-1}\}$, recall that  $a_n = m_n a_{n-1}$,
and for $k a_{n-1} \le x_2 \le (k+1)a_{n-1}$, with $0\le k \le m_n-1$, set
$$ f(x_1,x_2) = \frac{k+ h_{n-1}(x_1, x_2 - k a_{n-1})}{m_n}. $$
Then 
$$ \sE^{n-1}_{B_n}(f,f) 
= \sum_{Q \in \sS_{n-1}, Q \subset B_n} \sE^{n-1}_{Q}(f,f)
= m_n^2  \sE^{n-1}_{B_{n-1}}(h_{n-1},h_{n-1})m_n^{-2}=1. $$
If $K=0$ then we have $\mu^n_e \le \mu^{n-1}_e$, with strict inequality
for the edges in $D_n$. We thus have $\sigma^2_n(0) \le 1$. If 
we had equality, then the function $f$ would attain the minimum in
\eqref{e:vp}, and so would be harmonic in the environment $\mu_e^n$. 
Since this is not the case, we must have $\sigma_n^2(0) <1$.

To obtain a lower bound on $\sigma_n^2(K)$, we use the dual characterization
of effective resistance in terms of flows of minimal energy -- see
\cite{DS}, and \cite{BB3} for use in a similar context to this one.

Let $Q$ be a square in $\sQ_k$, with lower left corner $w=(w_1,w_2)$.
Let $Q'$ be the rectangle obtained by removing the top and bottom rows of $Q$:
$$ Q'= \{ (x_1,x_2): w_1 \le x_1 \le w_1+ k, w_1+1 \le x_2 \le w_1 + k-1\}. $$
A {\em flow} on $Q$ is an antisymmetric function $I$ on $Q \times Q$ which satisfies 
$I(x,y)=0$ if $x \not\sim y$,
$I(x,y)=-I(y,x)$, and 
$$ \sum_{y \sim x} I(x,y)=0 \q \text{ if } x \in Q'. $$
Let $\pd^+ Q =\{ (x_1, w_2+k): w_1 \le x_1 \le w_1+k \}$ be the top of $Q$.
The {\em flux} of a flow $I$ is
$$ F(I) = \sum_{x\in\pd^+ Q} \sum_{y \sim x} I(x,y). $$
For a flow $I$ and $m \in \{n-1,n\}$ set
$$ E^m_Q(I,I) = \half \sum_{x\in Q} \sum_{y\in Q} (\wt \mu^{Q,m}_{xy})^{-1} I(x,y)^2. $$
This is the energy of the flow $I$ in the electrical network given by $Q$ with 
conductances $(\wt \mu^{m,Q}_e)$.
If $\sI(Q)$ is the set of flows on $Q$ with flux 1, then
\bes
 \sigma_n(K)^{-2} = \inf\{ E^n_{B_n}(I,I): I \in \sI(B_n) \}. 
\ees
Let $I_{n-1}$ be the optimal flow for $\sigma^{-2}_{n-1}$. The square $B_n$ 
consists of $m_n^2$ copies of $B_{n-1}$; define a preliminary flow
$I'$ by placing a replica of $m_n^{-1} I_{n-1}$ in each of these copies.
For each square $Q \in \sQ_{a_{n-1}}$ with $Q\subset B_n$ we have
$E^{n-1}_Q(I',I')=m_n^{-2}$, and since there are $m_n^2$ of these squares
we have $E^{n-1}_{B_n}(I',I')=1$.

We now look at the tiling of $B_n$ by squares in $\sQ_{b_n}$; 
recall that $ \ell_n = a_n/b_n$ and that $\ell_n$ is an integer.
For each $Q \in \sQ_{b_n}$ we have $E^{n-1}_Q(I',I')= \ell_n^{-2}$. 
Label these squares by $(i,j)$ with $1\le i,j\le \ell_n$.

We now describe modifications to the flow $I'$ in a square $Q$.
Initially the flow runs from bottom to top of the square; if we reflect in the
diagonal of the square parallel to the line $x_1=x_2$, we obtain a flow $J$
which begins at the bottom, and emerges on the left side of the square.
As in \cite[Proposition 3.2]{BB3} we have $E_Q(J,J)\le E_Q(I',I')=\ell_n^{-2}$.
Thus `making a flow turn a corner' costs no more, in terms of energy, than letting
it run on straight.

Suppose we now consider the flow $I'$ in a column $(i_1, j), 1\le j \le \ell_n$,
and we wish to make the flow avoid an obstacle square $(i_1, j_1)$. Then we can make the
flow make a left turn in $(i_1, j_1-1)$, and then a right turn in  $(i_1-1, j_1-1)$
so that it resumes its overall vertical direction. 
This then gives rise to two flows in $(i_1-1, j_1-1)$: the original flow $I'$ plus
the new flow: as in \cite{BB3} the combined flow in the square $(i_1-1, j_1-1)$
has energy less than $4 \ell_n^{-2}$. If we carry the combined flow vertically
through the square $(i_1-1,j_1)$, and  make the similar modifications above
the obstacle, then we obtain overall a new flow $J'$ which matches $I'$
except on the 6 squares $(i,j), i_1\le i \le i_1, j_1-1\le j \le j_1+1$. 
The energy of the original flow in these 6 squares is $6\ell_n^{-2}$, while
the new flow will have energy less than $14\ell_n^{-2}$: we have a `cost'
of at most $4\ell_n^{-2}$ in the 3 squares $(i_1-1,j), j_1-1\le j \le j_1+1$, zero
in $(i_1,j_1)$ and  at most $\ell_n^{-2}$ in the two remaining squares. Thus the
overall energy cost of the diversion is at most $8 \ell_n^{-2}$ (see Fig.~\ref{fig3}).

\begin{figure} \includegraphics[width=8cm]{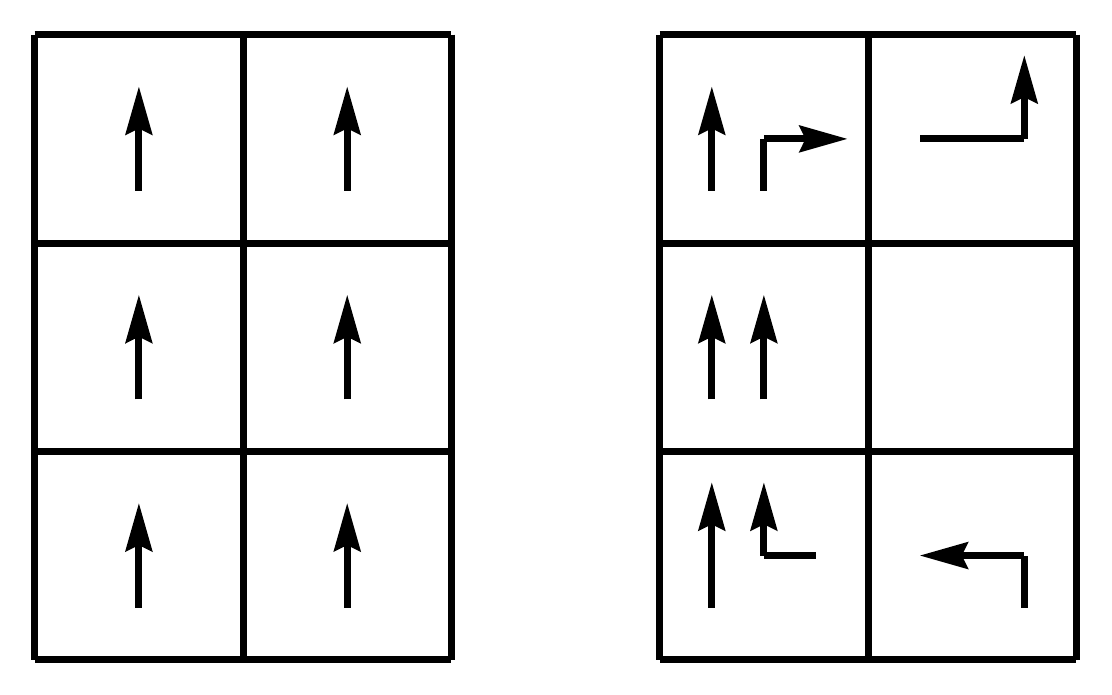}
\caption{Diversion of current around an obstacle square.}
\label{fig3}
\end{figure}

We now use a similar procedure to construct a modification of $I'$ in $B_n$
with conductances $(\mu_e^n)$. We have four obstacles, two oriented
vertically and resembling an $I$, and two horizontal ones. 
The crossbars on the $I$, that is the sets $D^{01}$, contain vertical edges with 
conductance $\eta_n \ll1 $. We therefore modify $I'$ to avoid these edges, and
the squares with side $b_n$ which contain them.

Consider the left vertical $I$, which has center $(a_n'-\beta_n, a_n')$.
Let $(i_1,j_1)$ be the square which contains at the top the bottom left branch
of the $I$, so that this square has top right corner $(a'_n-\beta_n, a'_n-10b_n)$.
The top of this square contains vertical edges with conductance $\eta_n$,
so we need to build a flow which avoids these. We therefore (as above) make the  
flow in the column $i_1$ take a left turn in square $(i_1,j_1-1)$, a right turn
in $(i_1-1,j_1-1)$, carry it vertically through $(i_1-1,j_1)$, take a right turn
in $(i_1-1,j_1+1)$ and carry it horizontally through $(i_1,j_1+1)$ into
the edges of high conductance at the right side of $(i_1,j_1+1)$.
The same pattern is then repeated on the other 3 branches of the left obstacle $I$,
and on the other vertical obstacle. 

We now bound the energy of the new flow $J$, and initially will make the calculations
just for the change in columns $i_1-1$ and $i_1$  below and to the left
of the point  $(a_n'-\beta_n, a_n')$. Write $M=10$ for the half of the overall
height of the obstacle.
There are $2(M+2)$ squares in this region 
where $I'$ and $J$ differ; these have labels $(i,j)$ with $i=i_1-1, i_1$ 
and $j_1-1\le j \le j_1+ M$. We begin by calculating the energy if $K=\infty$.
In 3 of these squares the new flow $J$ has energy
at most $4 \ell_n^{-2}$, in $M+1$ of them it has energy at most $\ell_n^{-2}$,
and in the remaining $M$ it has zero energy. So writing $R$ for this region we
have $E_R(I',I')= (2M+4)\ell_n^{-2} $, while
$$ E_R(J,J) \le (3\cdot 4 + M+1 )\ell_n^{-2} = (13+M)\ell_n^{-2}. $$
So 
\begin{align}\label{j27.6}
E_R(J,J) -E_R(I',I') \le (9-M) \ell_n^{-2} = - \ell_n^{-2}<0.
\end{align}
This is if $K=\infty$. Now suppose that $K<\infty$. The vertical edge in the obstacle carries a current
$2 /\ell_n$ and has height $M b_n$, so the energy of $J$ on these edges
is at most
\begin{align}\label{j27.7}
E'= \frac{4 \ell_n^{-2} M b_n}{K}\le \frac{4 M b_n}{Kn }. 
\end{align}
The last inequality holds because $\ell_n \geq \sqrt{n}$.
Finally it is necessary to modify $I'$ near the 4 ends of the two horizontal
obstacles. For this, we just modify $I'$ in squares of side $a_{n-1}$, and
arguments similar to the above show that for the new flow $J$ in this region $R'$, which
consists of $4+ 2 b_n/a_{n-1}$ squares of side $a_{n-1}$, we have 
\begin{align}\label{j27.8}
E_{R'}(J,J) - E_{R'}(I',I') \le \frac{9  b_n }{ a_{n-1} m_n^2} 
=  \frac{ 9 a_{n-1} }{ b_n } \ell_n^{-2}. 
\end{align}
The new flow $J$ avoids the edges where $\mu^n_e=\eta_n$.
Combining these terms we obtain for the whole square $B_n$, using \eqref{j27.6}-\eqref{j27.8},
\begin{align*}
E^n_{B_n}(J,J) - E^{n-1}_{B_n}(I',I') 
&\le - 8 \ell_n^{-2} + \frac{ 16 M b_n}{n K} +   \frac{ 40 a_{n-1} }{ b_n } \ell_n^{-2} \\
&\le - 7 \ell_n^{-2} + \frac{ 16 M b_n}{n K} < -\frac{7}{2n} + \frac{160 b_n}{nK}.
\end{align*}
So if $K' = 50 b_n$, we have
$$ \sigma_n^{-2}(K') \le  E^n_{B_n}(J,J) \le 1 - c n^{-1}< 1. $$
Hence there exists $K_n <50 b_n$ such that $\sigma_n^2(K_n)=1$.
\end{proof} 

\begin{lemma} \label{L:Lp}
Let $p<1$. Then $\bE \mu_e^p < \infty$, and $\bE \mu_e^{-p} <\infty$.
\end{lemma}

\begin{proof}
Since $\mu_e^n = \eta_n  =b_n^{-1-1/n}$ on a proportion $cb_n/a_n^2$ of the
edges in $B_n$, we have
\bes
\bE \mu_e^{-p} \le c \sum_n b_n^{p(1+1/n)}   \frac{b_n}{a_n^2} 
\le  c \sum_n b_n^{p+p/n -1}  < \infty.
\ees
Here we used the fact that $b_n \geq  2^{n}$. Similarly,
\bes
 \bE \mu_e^p \le c \sum_n K_n^p \frac{b_n}{a_n^2} 
\le  c \sum_n \frac{b_n^{1+p}}{a_n^2} < \infty.
\ees
\end{proof}

\begin{remark} \label{R:Emu}
A more accurate calculation for the upper bound on $\sigma^2(K)$ gives that we need 
$K_n > c b_n$ and consequently $\bE \mu_e =\infty$. Note that we also have
\be
 \limsup_{n \to \infty} n \bP( \mu_e > n)  =
 \limsup_{k \to \infty} b_k \bP( \mu_e > c b_k) = \lim_{k \to \infty} \frac{b_k^2}{a_k^2} =0. 
\ee 
\end{remark}

From now on we take $K_n$ to be such that $\sigma_n=1$ for all $n$.

\section{Weak invariance principle}  

Let $X=(X_t, t \in \bR_+, P^x_\om, x \in \bZ^d)$ be the process with 
generator \eqref{e:Ldef} associated with the environment $(\mu_e)$.
Recall \eqref{j27.4} and the definition of $X^n$, and define $X^{(n,\eps)}$ by
\bes
 X^{(n,\eps)}_t = \eps X^n_{\eps^2 t}, \, t \ge 0. 
\ees
Let $P^\om_n(\eps)$ be the law of $X^{(n,\eps)}$ on $\sD=\sD_1$, and
$P^\om(\eps)$ be the law of $X^{(\eps)}$.

Recall that the Prokhorov distance $\dP$
between probability measures on $\calD_1$ is defined as follows (see \cite[p.~238]{B}). 
For $A \subset \calD$, let $\calB(A,\eps) = \{x\in \calD: d_S (x, A) < \eps\}$. 
For probability measures $P$ and $Q$ on $\calD$, 
$\dP(P,Q) $ is the infimum of $\eps>0$ such that $P(A) \le Q(\calB(A,\eps)) + \eps$ and 
$Q(A) \le P(\calB(A,\eps)) + \eps$ for all Borel sets $A \subset \calD$.
Recall that convergence in the metric $\dP$ is equivalent to the weak convergence of measures.

To prove the WFCLT it is sufficient to prove:

\begin{theorem}
Let $\eps_n = 1/b_n$. Then $\bP  \lim_{n \to \infty} \dP( P^\om(\eps_n), \PBM) =0$.
\end{theorem}

\begin{proof}
Let $n \ge 1$ and suppose that $a_k, b_k$ have been chosen for $k \le n-1$.
By Proposition \ref{P:qnclt} we have for each $\om$ that 
$\dP( P^\om_{n-1}(\eps), \PBM) \to 0$. Note that the environment $\mu^{n-1}$ 
takes only finitely many values. So we can choose $b_n$ large enough so that
\be \label{e:an-choose1}
 \dP( P^\om_{n-1}(\eps), \PBM) < n^{-1} \q \hbox{ for  $0< \eps\le \eps_n$ and all $\om$.}
\ee

Now for $\lam>1$ set
$$ G(\lam) = \{ w \in \sD_1: \sup_{0\le s\le 1} |w(s)| \le \lam \}. $$
We have 
$$ \PBM(G(\lam)^c) \le  \exp( - c' \lam^2 ). $$

We can couple the processes $X^{n-1}$ and $X$ so that the two processes agree 
up to the first time $X^{n-1}$ hits the obstacle set $\bigcup_{k=n}^\infty D_k$.
Let $\xi_n(\om) =  \min\{ |x| : x \in \bigcup_{k=n}^\infty D_k(\om)\} $, and 
$$ F_n =\{ \xi_n > \lam b_n \}. $$
Let $m \ge n$, and consider the probability that $0$ is within a distance
$\lam b_n$ of $D_m$. Then $\sO_m$ has to lie in a set of area $c \lam b_n b_m$, and so
$$ \bP( \min_{x \in D_m} |x| \le \lam b_n ) \le \frac{c b_n b_m}{a_m^2} \le \frac{ c b_n}{ m b_m}. $$
Thus 
$$ \bP(F^c_n) \le c \sum_{m=n}^\infty  \frac{b_n}{ m b_m} 
\le \frac{c}{n}\Big(1  + \sum_{m=n+1}^\infty \frac{b_n}{b_m} \Big) \le  \frac{c'}{n}. $$

Suppose that $\om \in F_n$ and $n\geq 2$ so that $n^{-1} < \lambda/2$. Then using the 
coupling above, we have
\begin{align*}
 \dP( P^\om(\eps_n), P^\om_{n-1}(\eps_n) ) 
  &\le P^\om_0( \sup_{0\le s \le b_n^2} |X^{(n-1)}_s| > \lam b_n ) \\ 
&\le  \dP( P^\om_{n-1}(\eps_n), \PBM) +  \PBM(G(\lam/2)^c). 
\end{align*}

If now $\delta>0$, choose $\lam >1$ such that $\PBM(G(\lam/2)^c) < \delta/2$,
and then $N> 2/ \delta$ large enough so that $\bP(F_n^c) <  \delta $ for $n \ge N$.
Then combining the estimates above, if $n \ge N$ and $\om \in F_n$, 
$\dP(  P^\om(\eps_n), \PBM) < \delta$, so for $n \ge N$,
$\bP( \dP( P^\om(\eps_n), \PBM) > \delta) \le \bP(F_n^c)  < \delta, $
which proves the convergence in probability. 
\end{proof}

\section{Quenched invariance principle does not hold}\label{s:quen} 

 We will prove that the QFCLT does not hold for the processes
$X^{(\eps_n)}$,  and will argue by contradiction. 
If the QFCLT holds for $X$ with limit $\Sigma W$ then
since the WFCLT holds for $X^{(\eps_n)}$ with diffusion constant 1, $\Sigma$ must be the identity. 

Let $w^0_n=( a'_n - 10 b_n -1, a'_n-\beta_n)$ be the centre point on the left edge 
of the lowest of the four $n$-th level obstacles in the set $D^0_n$, 
and let $z^0_n=w_n - (\tfrac12 b_n,0)$.  Thus $z^0_n$ is situated a distance
$\half b_n$ to the left of $w^0_n$ -- see Fig.~\ref{fig4}.
Let 
\begin{align*}
 H_n^0(\lam) = B_\infty( z^0_n, \lam b_n) , \quad
 H_n(\lam) = \bigcup_{x\in a_n\bZ^2} ( x+ \sO_n + H_n^0(\lam)).
\end{align*}

\begin{figure} \includegraphics[width=8cm]{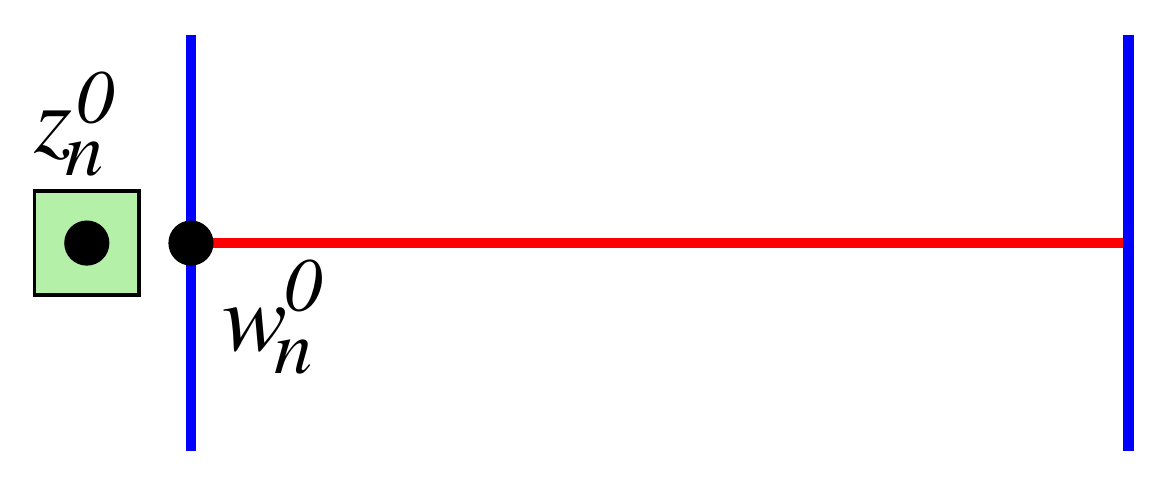}
\caption{The square represents $H_n^0(\frac{1}{8})$.}
\label{fig4}
\end{figure}

\begin{lemma} \label{L:hn}
For $\lam>0$ the event $\{0 \in H_n(\lam) \}$ occurs for infinitely many $n$, $\bP$-a.s.
\end{lemma}

\begin{proof}
Let $\sG_k=\sigma(\sO_1, \dots \sO_k)$. Given the values of $\sO_1, \dots \sO_{n-1}$,
the r.v. $\sO_n$ is uniformly distributed over $m_n^2$ points, with spacing $a_{n-1}$, and 
has to lie in a square with side $2 \lam b_n $ in order for the event  $\{0 \in H_n(\lam) \}$
to occur. Thus approximately $( 2\lam b_n/a_{n-1})^2$ of these values of
$\sO_n$ will cause $\{0 \in H_n(\lam) \}$ to occur. So 
$$ \bP( 0 \in H_n(\lam) \mid \sG_{n-1} ) \ge c \frac{ ( 2\lam b_n/ a_{n-1})^2 }{ (a_n/a_{n-1})^2}
 = c' \frac{ b_n^2}{a_n^2} \ge \frac{c''}{n}. $$
The conclusion then follows from an extension of the
second Borel-Cantelli Lemma. 
\end{proof}

\begin{lemma} \label{L:62}
With $\bP$-probability 1, the event
$G_n(\lam) = \{ H_n(\lam) \cap (\bigcup_{m=n+1}^\infty D_m ) \neq \emptyset \}$ 
occurs for only finitely many $n$.
\end{lemma}

\begin{proof} 
Let $m>n$. Then as in the previous lemma, by considering possible
positions of $\sO_m$, we have
$$ \bP( H_n(\lam) \cap D_m \neq \emptyset ) \le c \frac{ b_m b_n}{a_m^2}
\le c \frac{ b_n}{b_m}. $$
Since $b_{m} \ge 2^m b_{m-1} > 2^m b_n$, 
$$ \bP \Big( H_n(\lam) \cap \Big\{\bigcup_{m=n+1}^\infty D_m \neq \emptyset \Big\} \Big) 
\le  \sum_{m=n+1}^\infty  c \frac{ b_n}{b_m} \le  c 2^{-n}, $$
and the conclusion follows by Borel-Cantelli. 
\end{proof}

\begin{lemma} \label{L:dontcross}
Suppose that $0 \in H_n(1/8)$ and $H_n(4) \cap \big(\bigcup _{m=n+1}^\infty D_m\big) =\emptyset$.
Write $X_t=(X^1_t, X^2_t)$, and let  
\bes
 F =\{ |X^2_t| \le 3b_n/4, |X^1_t| \le 2b_n, 0\le t\le b_n^2,  X^1_{b_n^2} > 3b_n/4 \}. 
\ees
Then there exists a constant $A_{n-1}=A_{n-1}(\eta_1, K_1, \dots \eta_{n-1}, K_{n-1})$ such that
\bes
P^0_\om( F) \le  c b_n^{-1/n} A_{n-1} \log A_{n-1}.
\ees
\end{lemma}

\proof Let $w_n=(x_n,y_n)$ be the element of $\{ w^0_n + \sO_n + a_n x, x\in \bZ^2\}$ which is
closest to 0. Then, under the 
hypotheses of the Lemma, we have $3 b_n/8 \le x_n \le 5 b_n/8$, and $|y_n| \le b_n/8$.
Thus the square $B_\infty(0,2b_n)$ intersects the obstacle set $D_n$, but does not intersect
$D_m$ for any $m>n$. Hence if $F$ holds then we can couple $X^n$ and $X$ so that
$X^n_t=X_t$ for $0 \le t\le b_n^2$.

Let $\bH=\{ (x,y): x \le x_n \}$, and $J=B \cap \pd_i \bH$.
If $F$ holds then $X^n$ has to cross the line $J$, and therefore has to cross
an edge of conductance $\eta_n$. 
Let $Y$ be the process with edge conductances $\mu'_e$, where
$\mu'_e=\mu^{n-1}_e$ except that $\mu'_e=0$ if $e=\{ (x_n,y), (x_n+1,y)\}$ for $y \in \bZ$.
Thus the line $\pd_i \bH$ is a reflecting barrier for $Y$. Let
$$ L_t = \int_0^t 1_{ (Y_s \in J) }ds $$ 
be the amount of time spent by $Y$ in $J$, and
$$ G= \{ |Y^2_t| \le 3b_n/4, |Y^1_t| \le 2b_n, 0\le t\le b_n^2 \}.$$
Assuming that $G$ holds, let $\xi_1$ be a standard
${\tt exp(1)}$ r.v., set $T= \inf\{s: L_s > \xi_1/\eta_n \}$, and let $X^n_t=Y_t$ on $[0,T)$, and
$X^n_T = Y_T + (1,0)$. 
Note that one can complete the definition of $X^n_t$ for $t\geq T$ in such a way that
the process $X^n$ has the same distribution as the process defined by \eqref{j27.4}.
We have
$$ P^0_\om( G \cap \{ X^n_s = Y^n_s, 0\le s \le b_n^2 \}) = 
 E^0_\om( 1_G \exp( - \eta_n L_{b_n^2} ) ). $$
So
$$ P^0_\om( G \cap \{ T \le b_n^2 \}) =  E^0_\om( 1_G (1-\exp( - \eta_n L_{b_n^2} ) )
\le  E^0_\om( 1_G \eta_n L_{b_n^2}) \le \eta_n E^0_\om L_{b_n^2}. $$
The process $Y$ has conductances bounded away from 0 and infinity on $\bH$, 
so by \cite{D1} $Y$ has a transition probability $p_t(w,z)$ which satisfies
$$ p_t(w,z) \le A t^{-1} \exp( A^{-1} |w-z|^2/t ), \, w,z \in \bH, \, t \ge |w-z|. $$
In addition if $r=|w-z|\ge A$ then $p_t(w,z) \le p_r(w,z)$. 
Here $A=A_{n-1}$ is a possibly large constant which depends on $(\eta_i, K_i, 1\le i \le n-1)$.
We can take $A \ge 10$. For $w\in J$ we have $|w| \ge b_n/4$ and so provided $b_n \ge 8A$,
\begin{align*}
E^0_\om \int_0^{b_n^2} 1_{( Y_s =w )} ds 
&= \int_0^{b_n^2} p_t(0,w) dt 
\le b_n p_{b_n} (0,w) + \int_{b_n}^{b_n^2}  p_t(0,w) dt \\
&\le c A e^{- b_n/A} + A \int_{0}^{b_n^2} t^{-1} \exp( -b_n^2/16 At )dt  \le c A \log (A). 
\end{align*}
So since $|J| \le 2b_n$, 
\bes
 P^0_\om( G \cap \{ T \le b_n^2 \} )\le  c \eta_n b_n A \log A \le c b_n^{-1/n} A \log A.
\ees
Finally, the construction of $X^n$ from $Y$ gives that
$P^0_\om(F) \le  P^0_\om( G \cap \{ T \le b_n^2 \} )$.
\qed

\begin{proof}[Proof of Theorem \ref{T:main1}(b).]
We now choose $b_n$ large enough so that for all $n\ge 2$,
\be \label{e:bncond1}
b_n^{-1/n} A_{n-1} \log A_{n-1} < n^{-1}.
\ee

Let $W_t = (W^1_t, W^2_t)$ denote 2-dimensional Brownian motion with $W_0=0$,
and let $\PBM$ denote its distribution. 
For a 2-dimensional process $Z=(Z^1, Z^2)$, define the event 
\begin{align*}
F(Z)  = \Big \{  |Z^2_s| < 3/4, |Z^1_s| \le 2, 0\le s \le 1,  Z^1_1 > 1 \Big \}. 
\end{align*}
The support theorem implies that $p_1 := \PBM(F(W)) >0$.
Write $F_n = F(X^{(\eps_n)})$. 

Let $N_1 = N_1(\om)$ be such that the event $G_n(4)$ defined in Lemma
\ref{L:62} does not occur for $n \ge N_1$.
Let $\Lambda=\Lambda(\om)$ be the set of $n > N_1$ such that
$0 \in H_n(\tfrac18)$. Then $\bP(\Lambda \hbox{ is infinite})=1$ by  Lemma  \ref{L:hn}.
By Lemma \ref{L:dontcross} and the choice of $b_n$ in \eqref{e:bncond1} 
we have $ P^0_\om( F_n)  <  cn^{-1}$ for $n\in\Lambda$. So 
$$ P^0_\om( F_n ) \to 0 \hbox{ as  $n \to \infty$ with $n \in \Lambda$} . $$
Thus whenever $\Lambda(\om)$ is infinite the sequence of processes 
$ (X^{(\eps_n)}_t, t \in [0,1], P^0_\om), \, n \ge 1, $
cannot converge to $W$, and the QFCLT therefore fails.   
\end{proof}

\begin{remark} \label{R:Zd}
We can construct similar obstacle sets in $\bZ^d$ with $d \ge 3$, and 
we now outline briefly the main differences from the $d=2$ case.

We take $b_n = a_n n^{-1/d}$, so that $\sum b_n^d/a_n^d =\infty$, and the analogue 
of Lemma \ref{L:62} holds. 
In a cube side $a_n$ we take $2d$ obstacle sets, arranged in symmetric fashion 
around the centre of the cube. Each obstacle has an associated `direction' $i \in \{1, \dots d\}$.
An obstacle of direction $i$ consists of a $2 b_n^{d-1}$ edges of
low conductance $\eta_n$, arranged in two $d-1$ dimensional `plates' a distance 
$M b_n$ apart, with each edge in the direction $i$. The two plates are connected by 
$d-1$ dimensional plates of high conductance $K_n$.
Thus the total number of edges in the obstacles is $c b_n^{d-1}$, so taking $a_n/a_{n-1}$
large enough, we have $\sum b_n^{d-1}/a_n^d<\infty$, and the same arguments as in 
Section \ref{const} show that the environment is well defined, stationary and ergodic.

The conductivity across a cube side $N$ in $\bZ^d$ is $N^{d-2}$.  Thus if we write 
$\sigma^2_n(\eta_n, K_n)$ for the limiting diffusion constant of the process $X^n$, 
and $R_n=R_n(\eta_n,K_n)$ for the effective resistance across a cube side $a_n$, then
\eqref{e:vp} is replaced by:
\be \label{e:diff-d}
 \sigma_n^2(\eta_n, K_n) = a_n^{2-d} R_n^{-1}. 
\ee
For the QFCLT to fail, we need $\eta_n = o(b_n^{-1})$, as in the two-dimensional
case. With this choice we have $R_n(\eta_n, 0)^{-1} < a_n^{d-2}$, and as in
Theorem \ref{T:eK} we need to show that if $K_n$ is large enough then 
$R_n(\eta_n, K_n)^{-1} > a_n^{d-2}$.

Recall that $\ell_n=a_n/b_n$. 
Let $I'$ be as in Theorem \ref{T:eK}; then $I'$ has flux $\ell_n^{-d+1}$ across each
sub-cube $Q'$ of side $b_n$. If the sub-cube does not intersect the obstacles at level
$n$, then $E_{Q'}(I',I')= \ell_n^{-d} a_n^{2-d}$. 
The `cost' of diverting $I'$ around a low conductance obstacle is therefore of order
$c \ell_n^{-d} a_n^{2-d}= c b_n^{-d+2} \ell_n^{-2d+2}$ -- see \cite{McG}.
As in Theorem \ref{T:eK} we divert  the flow onto the regions of high conductance, 
so as to obtain some cubes in which the new flow has zero energy.
To estimate the energy in the high conductance bonds, note that
we have $2(d-1)b_n^{d-2}$ sets of parallel paths of edges of high conductance, and
each path is of length $M b_n$, so the flow in each edge is
$F_n= \ell_n^{-d+1}/ b_n^{d-2}(2d-2)$. Hence the total energy dissipation in the 
high conductance edges is
$$ K^{-1} M F_n^2=  \frac{ c'K^{-1} M b_n^{d-1}}{ \ell_n^{2d-2} b_n^{2d-4}} 
  = \frac{c'K^{-1} M }{ \ell_n^{2d-2} b_n^{d-3}}. $$
We therefore need
$$  \frac{c'K^{-1}  M }{ \ell_n^{2d-2} b_n^{d-3}} < \frac{c}{ b_n^{d-2} \ell_n^{2d-2}}, $$
that is we need to choose $K_n > c M b_n$ for some constant $c$.  Since
$$ \bE \mu_e^p \asymp \sum_n \frac{K_n^p b_n^{d-1}} {a_n^d} 
\asymp M \sum_n \frac{b_n^{d-1+p}}{a_n^d}, $$
we find that in $d\ge 3$ our example also has $\bE \mu_e^{\pm p}<\infty$ if and only if $p<1$.
\end{remark}

\end{document}